\documentclass{conm-p-l}

\usepackage{amsmath}
\usepackage{amsthm}
\usepackage{amssymb}
\usepackage[all]{xy}

\DeclareMathSymbol{\rightrightarrows}  {\mathrel}{AMSa}{"13}

\def\Shv{\mathbf{Shv}}
\def\Tot{\operatorname{Tot}}

\def\Pre{\operatorname{Pre}}

\def\Ex{\operatorname{Ex}}

\def\Sp{\operatorname{Sp}}

\catcode`\@=11
\def\varholim@#1#2{\mathop{\vtop{\ialign{##\crcr
 \hfil$#1\m@th\operator@font holim$\hfil\crcr
 \noalign{\nointerlineskip\kern\ex@}#2#1\crcr
 \noalign{\nointerlineskip\kern-\ex@}\crcr}}}}
\def\hocolim{\mathpalette\varholim@\rightarrowfill@} 
\def\hoinvlim{\mathpalette\varholim@\leftarrowfill@}
\catcode`\@=\active

\newtheorem{theorem}{Theorem}
\newtheorem{lemma}[theorem]{Lemma}
\newtheorem{corollary}[theorem]{Corollary}

\theoremstyle{definition}

\newtheorem{example}[theorem]{Example}

\newtheorem{remark}[theorem]{Remark}

\begin{document}

\title{Galois descent criteria}
                                                                                \author{J.F. Jardine}
 
 

\thanks{This research was supported by NSERC.}                                                                                
 
\maketitle

\begin{abstract}
\noindent
This paper gives an introduction to homotopy descent, and its applications in algebraic $K$-theory computations for fields. On the \'etale site of a field, a fibrant model of a simplicial presheaf can be constructed from naive Galois cohomological objects given by homotopy fixed point constructions, but only up to pro-equivalence. The homotopy fixed point spaces define finite Galois descent for simplicial presheaves (and their relatives) over a field, but a pro-categorical construction is a necessary second step for passage from finite descent conditions to full homotopy descent in a Galois cohomological setting.
\end{abstract}

\section*{Introduction}

Descent theory is a large subject, which appears in many forms in geometry, number theory and topology.

Initially, it was a set of methods for constructing global features of
a ``space'' from a set of local data that satisfies patching
conditions, or for defining a variety over a base field from a variety
over a finite separable extension that comes equipped with some type
of cocycle. The latter field of definition problem appears in early
work of Weil \cite{Weil}; it was later subsumed by a general approach
of Grothendieck \cite{FGA} in the theory of faithfully flat descent.

The early descriptions of patching conditions were later generalized
to isomorphisms of structures on patches which are defined up to coherent
isomorphism, in the formulation of the notion of effective descent
that one finds in the theory of stacks and, more generally, higher
stacks \cite{Simpson-descent}.
\medskip

Cohomological descent is a spectral sequence
technique for computing the cohomology of a ``space'' $S$ from the
cohomology of the members of a covering. The theory is discussed in
detail in SGA4, \cite[Exp. Vbis]{SGA4-2}, while the original
spectral sequence for an ordinary covering was introduced by Godement \cite{Godement}.

The construction of the descent spectral sequence for a covering $U \to S$ (sheaf epimorphism) starts with a \v{C}ech resolution $\check{C}(U) \to S$ for the covering and an injective resolution $A \to I^{\bullet}$ of a coefficient abelian sheaf $A$. One forms the third quadrant bicomplex
\begin{equation}\label{eq 1}
  \hom(\check{C}(U),I^{\bullet}),
\end{equation}
and the resulting spectral sequence converges to sheaf cohomology $H^{\ast}(S,A)$, with the form
\begin{equation}\label{eq 2}
  E_{2}^{p,q}=H^{p}(H^{q}(\check{C}(U),A)) \Rightarrow H^{p+q}(S,A).
\end{equation}

It is a more recent observation that one recovers $H^{\ast}(S,A)$ from the bicomplex (\ref{eq 1}) since the resolution $\check{C}(U) \to S$ is a stalkwise weak equivalence of simplicial sheaves. The \v{C}ech resolution $\check{C}(U)$ is a simplicial object which is made up of the components of the covering $U$ and their iterated intersections.

The variation of the descent spectral sequence that is discussed in
SGA4 is constructed by replacing the \v{C}ech resolution by a
hypercover $V \to S$. In modern terms, a hypercover is a local trivial
fibration of simplicial sheaves, but such a map was initially defined to be a simplicial scheme over $S$
which satisfied a set of local epimorphism conditions defined by its
coskeleta \cite{AM}.

The key observation for these constructions is that, if $X \to Y$ is a stalkwise weak equivalence of simplicial sheaves (or presheaves), then the induced map of bicomplexes
\begin{equation}\label{eq 3}
  \hom(Y,I^{\bullet}) \to \hom(X,I^{\bullet})
\end{equation}
induces a cohomology isomorphism of total complexes. Thus, one has a definition
\begin{equation*}
  H^{n}(X,A) := H^{n}(\Tot\hom(X,I^{\bullet}))
\end{equation*}
of the cohomology of a simplicial presheaf $X$ with coefficients in an abelian sheaf $A$ that is independent of the stalkwise homotopy type of $X$, along with a spectral sequence that computes it. 
\medskip

Descent theory became a homotopy theoretic pursuit with the
introduction of local homotopy theories for simplicial presheaves and
sheaves, and presheaves of spectra. These homotopy theories evolved from ideas of Grothendieck; their formalization essentially began with Illusie's thesis \cite{Ill01}.

Local homotopy theories are Quillen model structures: a local weak
equivalence of simplicial presheaves or sheaves is a map which induces
weak equivalences at all stalks, and a cofibration is a
monomorphism. The local homotopy theory of presheaves of spectra is
constructed from the homotopy theory of simplicial presheaves by using
methods of Bousfield and Friedlander \cite{BF}. The fibrations for these
theories are now commonly called injective fibrations.

In the setup for the cohomological descent spectral sequence
(\ref{eq 2}), the injective resolution $I^{\bullet}$ ``satisfies descent'',
in that it behaves like an injective fibrant object, with the result
that a local weak equivalence $X \to Y$ induces a quasi-isomorphism
(\ref{eq 3}). Homotopical descent theory is the study of simplicial
objects and spectrum objects that are nearly injective fibrant.
\medskip

One says that a simplicial presheaf $X$ {\it satisfies descent} (or {\it homotopy descent}) if any local weak equivalence $X \to Z$ with $Z$ injective fibrant is a sectionwise weak equivalence, in the sense that the maps $X(U) \to Z(U)$ are weak equivalences of simplicial sets for all objects $U$ in the underlying site.

This form of descent is a statement about the sectionwise behaviour of simplicial {\it presheaves}, or {\it presheaves} of spectra, and is oriented towards computing homotopy groups in sections. The role of sheaves is incidental, except in the analysis of local behaviour.

There are many examples:
\medskip

\noindent
1)\
Every local weak equivalence $Z \to Z'$ of injective fibrant objects is a sectionwise equivalence by formal nonsense (discussed below), so that all injective fibrant objects satisfy descent.
\medskip

\noindent
2)\
One can show \cite[Sec. 9.2]{LocHom} that a sheaf of groupoids $G$ is a stack (i.e. satisfies the effective descent condition) if and only if its nerve $BG$ satisfies descent.
\medskip

The advantage of having an object $X$ which satisfies descent is that there are machines (e.g. Postnikov tower, or Godement resolution) that can be used to produce a spectral
sequence 
\begin{equation}\label{eq 4}
  E_{2}^{s,t} = H^{s}(S,\tilde{\pi}_{t}X)\ ``\Rightarrow"\ \pi_{t-s}(X(S))
\end{equation}
which computes the homotopy groups of the space $X(S)$ in global sections from sheaf cohomology for $S$ with coefficients in the homotopy group sheaves of $X$.  This is the homotopy descent spectral sequence. 

The spectral sequence (\ref{eq 4}) is a Bousfield-Kan spectral sequence for a tower of fibrations, so convergence can be a problem, and there may also be a problem with knowing what it converges to. Both issues are circumvented in practice by insisting on a global bound on cohomological dimension --- see Section 4.
\medskip

The availability of a calculational device such as (\ref{eq 4}) for objects $X$ which satisfy descent means that the hunt is on for such objects, for various topologies and in different contexts.

The algebraic $K$-theory presheaf of spectra $K$, for example, satisfies descent
for the Nisnevich topology on the category $Sm\vert_{S}$ of smooth $S$-schemes, where $S$ is a regular Noetherian scheme of finite dimension. This follows from
the existence of localization sequences in $K$-theory for such schemes, so that the $K$-theory presheaf satisfies a ``$cd$-excision'' property.

A general result of Morel and Voevodsky \cite{MV}, \cite[Thm. 5.39]{LocHom} says that any simplicial presheaf on $Sm\vert_{S}$ that satisfies the $cd$-excision property satisfies Nisnevich descent. The proof of the Morel-Voevodsky theorem is based on an earlier theorem of Brown and Gersten, which gives a descent criterion for simplicial presheaves on the standard site of open subsets of a Noetherian topological space. The descent criterion of Brown-Gersten amounts to homotopy cartesian patching for pairs of open subsets.

The arguments for the Morel-Voevodsky and Brown-Gersten descent theorems are geometric and subtle, and depend strongly on the ambient Grothen\-dieck topologies. Descent theorems are interesting and important geometric results,
and finding one of them is a major event. 
\medskip

Homotopy descent problems originated in algebraic $K$-theory, in the complex of problems related to the Lichtenbaum-Quillen conjecture.

Suppose that $k$ is a field, that $\ell$ is a prime number which is distinct from the characteristic of $k$.
The mod $\ell$ algebraic $K$-theory presheaf of spectra $K/\ell$ on smooth $k$-schemes is the cofibre of multiplication by $\ell$ on the algebraic $K$-theory presheaf $K$, and the stable homotopy groups $\pi_{p}K/\ell(k)$ are the mod $\ell$ $K$-groups $K_{p}(k,\mathbb{Z}/\ell)$ of the field $k$. The presheaf of spectra $K/\ell$ has an injective fibrant model $j: K/\ell \to LK/\ell$ for the \'etale topology on $k$, and the stable homotopy groups $\pi_{p}LK/\ell(k)$ are the \'etale $K$-groups $K_{p}^{et}(k,\mathbb{Z}/\ell)$ of $k$. The map $j$ induces a comparison
\begin{equation}\label{eq 5}
  K_{p}(k,\mathbb{Z}/\ell) \to K_{p}^{et}(k,\mathbb{Z}/\ell)
\end{equation}
in global sections, and the Lichtenbaum-Quillen conjecture asserts that this map is an isomorphism in the infinite range of degrees above the Galois $\ell$-cohomological dimension of $k$, which dimension is assumed to be finite.

The point of this conjecture is that algebraic $K$-theory with torsion coefficients should be computable from \'etale (or Galois) cohomology. At the time that it was formulated, the conjecture was a striking leap of faith from calculations in low degrees. The precise form of the conjecture that incorporates the injective fibrant model $j: K/\ell \to LK/\ell$ followed much later.

Thomason's descent theorem for Bott periodic $K$-theory \cite{AKTEC} was a first approximation to Lichtenbaum-Quillen. His theorem says that formally inverting the Bott element $\beta$ in $K_{\ast}(k,\mathbb{Z}/\ell)$ produces a presheaf of spectra $K/\ell(1/\beta)$ which satisfies descent for the \'etale topology on the field $k$. \'Etale $K$-theory is Bott periodic, so that the spectrum object $K/\ell(1/\beta)$ is a model for the \'etale $K$-theory presheaf.

The Lichtenbaum-Quillen conjecture was proved much later --- it is a consequence of the Bloch-Kato conjecture, via the Beilinson-Lichtenbaum conjecture \cite{SV2}, while Voevodsky's proof of Bloch-Kato appears in \cite{Voev4}.
\medskip

Voevodsky's work on Bloch-Kato depended on the introduction and use of motivic techniques, and was a radical departure from the methods that were used in attempts to calculate the $K$-theory of fields up to the mid 1990s.

Before Voevodsky, the general plan for showing that the \'etale descent spectral sequence converged to the algebraic $K$-theory of the base field followed the methods of Thomason, and in part amounted to attempts to mimic, for $K$-theory, the observation that the Galois cohomology of a field $k$ can be computed from \v{C}ech cohomology.
By the time that Thomason's paper \cite{AKTEC} appeared, the $E_{2}$-term of the \'etale descent spectral sequence for the $K$-theory of fields was known from Suslin's calculations of the $K$-theory of algebraically closed fields \cite{Su1}, \cite{Su2}.

In modern terms, the relationship between Galois cohomology and \v{C}ech cohomology for a field $k$ has the form of an explicit isomorphism
\begin{equation}\label{eq 6}
  H^{p}_{Gal}(k,A) \xrightarrow{\cong} \varinjlim_{L/k}\ H^{p}\hom(EG \times_{G} \Sp(L),A),
  \end{equation}
which is defined for any abelian sheaf $A$ on the \'etale site for $k$. Here, $L$ varies through the finite Galois extensions of $k$, and we write $G = Gal(L/k)$ for the Galois group of such an extension $L$. Here, the scheme $\Sp(L)$ is the Zariski spectrum of the field $L$.

The simplicial sheaf $EG \times_{G} \Sp(L)$ is the Borel construction for the action of $G$ on the \'etale sheaf represented by the $k$-scheme $\Sp(L)$, and is isomorphic to
the \v{C}ech resolution for the \'etale cover $\Sp(L) \to \Sp(k)$.

The complex $\hom(EG \times_{G} \Sp(L),A)$ has $n$-cochains given by
\begin{equation*}
  \hom(EG \times_{G} \Sp(L),A)^{n} = \prod_{G^{\times n}}\ A(L),
\end{equation*}
and is the homotopy fixed points complex for the action of $G$ on the abelian group $A(L)$ of $L$-points of $A$.

It is a critical observation of Thomason that if $B$ is an abelian presheaf which is additive in the sense that it takes finite disjoint unions of schemes to products, then there is an isomorphism
\begin{equation}\label{eq 7}
  H^{p}_{Gal}(k,\tilde{B}) \cong \varinjlim_{L/k}\ H^{p}\hom(EG \times_{G} \Sp(L),B),
  \end{equation}
which computes cohomology with coefficients in the associated sheaf $\tilde{B}$ from the presheaf-theoretic cochain complexes $\hom(EG \times_{G} \Sp(L),B)$.

The $K$-theory presheaf of spectra $K/\ell$ is additive, and it's still a leap, but one could hope that the analogous comparison map of spectra
\begin{equation}\label{eq 8}
  K/\ell(k) \to \varinjlim_{L/k}\ \mathbf{hom}(EG \times_{G} \Sp(L),K/\ell)
\end{equation}
induces an isomorphism in stable homotopy groups in an appropriate range, and that the colimit on the right would be equivalent to the mod $\ell$ \'etale $K$-theory spectrum of the field $k$.

There were variations of this hope. The map (\ref{eq 8}) is a colimit of the comparison maps
\begin{equation}\label{eq 9}
  K/\ell(k) \to \mathbf{hom}(EG \times_{G} \Sp(L),K/\ell),
\end{equation}
and one could ask that each such map induces an isomorphism in homotopy groups in an appropriate range.

The function complex spectrum $\mathbf{hom}(EG \times_{G} \Sp(L),K/\ell)$ is the homotopy fixed points spectrum for the action of the Galois group $G$ on the spectrum $K/\ell(L)$, and the question of whether or not (\ref{eq 9}) is a weak equivalence is commonly called a homotopy fixed points problem. It is also a finite descent problem.
\medskip

There were many attempts to solve homotopy fixed points problems for algebraic $K$-theory in the pre-motives era, with the general expectation that the question of identifying the colimit in (\ref{eq 8}) with the \'etale $K$-theory spectrum should then take care of itself.

The identification problem, however, turned out to be hard. Attempts
to address it invariably ended in failure, and always involved the ``canonical mistake'', which is the false assumption that inverse limits commute with filtered colimits.
\medskip

It is a technical application of the methods of this paper that the identification of the colimit
\begin{equation*}
\varinjlim_{L/k}\ \mathbf{hom}(EG \times_{G} \Sp(L),K/\ell)  
\end{equation*}
with the \'etale $K$-theory spectrum
cannot work out, except in a suitable pro category.

This is expressed in more abstract terms as Theorem \ref{th 24} below for a certain class of simplicial presheaves on the \'etale site for $k$. The mod $\ell$ $K$-theory presheaf of spectra $K/\ell$, or rather its component level spaces $(K/\ell)^{n}$, are examples of such objects.
\medskip


The main body of this paper is set in the context of simplicial presheaves and sheaves on the site $G-\mathbf{Set}_{df}$ of discrete finite $G$-sets for a profinite group $G$ and their $G$-equivariant maps.  The coverings for this site are the surjective maps.

Explicitly, a profinite group is a functor $G: I \to \mathbf{Grp}$, with $i \mapsto G_{i}$ for objects $i$ in the index category $I$. The category $I$ is small and left filtered, and the functor $G$ takes values in finite groups. We also require (following Serre \cite{Serre-CG2}) that all transition morphisms $G_{i} \to G_{j}$ are surjective. All Galois groups have these general properties.

If $k$ is a field, then the finite \'etale site is equivalent to the site $G-\mathbf{Set}_{df}$ for the absolute Galois group $G$ of $k$, via imbeddings of finite separable extensions of $k$ in its algebraic closure.

Until one reaches the specialized calculations of Section 4, everything that is said about simplicial presheaves and presheaves of spectra on \'etale sites of fields is a consequence of general results about the corresponding objects associated to the sites $G-\mathbf{Set}_{df}$ for profinite groups $G$. 

The local homotopy theory for general profinite groups was first explicitly described by Goerss \cite{Goerss1}, and has since become a central structural component of the chromatic picture of the stable homotopy groups of spheres.

This paper proceeds on a separate track, and reflects the focus on generalized Galois cohomology and descent questions which arose in algebraic $K$-theory, as partially described above. See also \cite{GECT}. 
\medskip

Some basic features of the local homotopy theory for profinite groups are recalled in Section 1. We shall also use results about cosimplicial spaces that are displayed in Section 2.

With this collection of techniques in hand, we arrive at the following:

\begin{theorem}\label{th 1}
Suppose that $f: X \to Y$ is a local weak equivalence between
presheaves of Kan complexes on the site $G-\mathbf{Set}_{df}$ such that $X$
and $Y$ have only finitely many non-trivial presheaves of homotopy
groups. Then the induced map
\begin{equation*}
f_{\ast}: \varinjlim_{i}\ \mathbf{hom}(EG_{i} \times_{G_{i}} G_{i},X) \to 
\varinjlim_{i}\ \mathbf{hom}(EG_{i} \times_{G_{i}} G_{i},Y)
\end{equation*}
is a weak equivalence.
\end{theorem}

I say that a presheaf of Kan complexes $X$ has {\it only finitely many non-trivial presheaves of homotopy groups} if the canonical map 
\begin{equation*}
  p: X \to \mathbf{P}_{n}X
\end{equation*}
is a sectionwise weak equivalence
for some $n$, where $\mathbf{P}_{n}X$ is a Postnikov section of $X$.
We can also say, more compactly, that $X$ is a {\it sectionwise $n$-type}.

Theorem \ref{th 1} appears as Theorem \ref{th 9} below. It has the following special case:

\begin{corollary}\label{cor 2}
Suppose that $f: X \to Y$ is a local weak equivalence between
presheaves of Kan complexes on the finite \'etale site of a field $k$
 such that $X$
and $Y$ are sectionwise $n$-types. Then the induced map
\begin{equation*}
f_{\ast}: \varinjlim_{L/k}\ \mathbf{hom}(EG \times_{G} \Sp(L),X) \to 
\varinjlim_{L/k}\ \mathbf{hom}(EG \times_{G} \Sp(L),Y)
\end{equation*}
is a weak equivalence.
\end{corollary}  

Theorem \ref{th 1} is proved by inductively solving obstructions for cosimplicial spaces after refining along the filtered diagram associated to the profinite group $G$, by using methods from Section 2. The assumption that the simplicial presheaf $X$ has only finitely many non-trivial presheaves of homotopy groups means that the obstructions can be solved in finitely many steps.

It is important to note that if $X$ is a sectionwise $n$-type and if $j: X \to Z$ is an injective fibrant model, then $Z$ is a sectionwise $n$-type. This observation is general \cite{GECT}, and is used in the proof of Corollary \ref{cor 10} below. 

When specialized to the fields case, Theorem \ref{th 1} implies that the colimit
\begin{equation*}
  \varinjlim_{L/k}\ \mathbf{hom}(EG \times_{G} \Sp(L),X)
\end{equation*}
is weakly equivalent to the simplicial set $Z(k)$ of global sections of a fibrant model $j: X \to Z$ on the finite \'etale site of a field $k$, provided that $X$ is a sectionwise $n$-type.

In particular, if $X$ is a sectionwise $n$-type,  and if $X$ also satisfies finite descent, then the map $X(k) \to Z(k)$ in global sections is a weak equivalence.
\medskip

Generally, Theorem \ref{th 1} means that one can use Galois cohomological methods to construct injective fibrant models for simplicial presheaves $X$ having finitely many non-trivial presheaves of homotopy groups. This construction specializes to (and incorporates) the identification (\ref{eq 6}) of Galois cohomology with \v{C}ech cohomology.

Going further involves use of the homotopy theory of pro-objects and their local pro-equivalences, which is enabled by \cite{pro}.

In general, a simplicial presheaf $Y$ is pro-equivalent to its derived Postnikov tower, via the canonical map $Y \to \mathbf{P}_{\ast}Y$. The Postnikov tower $\mathbf{P}_{\ast}Y$ has a (naive) fibrant model $\mathbf{P}_{\ast}(Y) \to L\mathbf{P}_{\ast}(Y)$ in the model category of towers of simplicial presheaves. One then has a string of local pro-equivalences
\begin{equation*}
  Y \to \mathbf{P}_{\ast}Y \to L\mathbf{P}_{\ast}Y,
\end{equation*}
and it follows from Corollary \ref{cor 2} that the induced composite in global sections is pro-equivalent to the pro-map
\begin{equation*}
  \theta: Y(k) \to \varinjlim_{L/k}\ \mathbf{hom}(EG \times_{G} \Sp(L),\mathbf{P}_{\ast}Y).
  \end{equation*}
There are two questions:
\begin{itemize}
\item[1)] Is the displayed map $\theta$ a pro-equivalence?
\item[2)] If $j: Y \to Z$ is an injective fibrant model for $Y$, is the corresponding map
  \begin{equation*}
  \theta: Z(k) \to \varinjlim_{L/k}\ \mathbf{hom}(EG \times_{G} \Sp(L),\mathbf{P}_{\ast}Z).
  \end{equation*}
  a pro-equivalence?
\end{itemize}

If the answer to both questions is yes, then the map $Y(k) \to Z(k)$ is a pro-equivalence of spaces, and hence a weak equivalence by Corollary \ref{cor 13} of this paper. These questions are the Galois descent criteria for a simplicial presheaf $Y$ and its fibrant model $Z$ on the \'etale site of a field.

Question 2) is non-trivial, perhaps surprisingly, but one should observe that the Postnikov tower construction $\mathbf{P}_{\ast}Z$ does not preserve injective fibrant objects.

The imposition of a global bound on cohomological dimension forces a positive answer to question 2), by Lemma \ref{lem 22} of this paper, and in that case
the simplicial presheaf $Y$ satisfies Galois descent if and only if the map $\theta$ is a pro-equivalence. 

Such bounds on global cohomological dimension are commonly met in geometric applications, including the Galois descent problem for algebraic $K$-theory with torsion coefficients.
\vfill\eject

\tableofcontents

\section{Profinite groups}

We begin with a discussion of some generalities about profinite
groups, in order to establish notation.

Suppose that the group-valued functor $G: I \to \mathbf{Grp}$ is a
profinite group. This means that $I$ is left filtered (any two objects
$i,i'$ have a common lower bound, and any two morphisms $i
\rightrightarrows j$ have a weak equalizer), and that all of the
constituent groups $G_{i}$, $i \in I$, are finite. We shall also assume
that all of the transition homomorphisms $G_{i} \to G_{j}$ in the
diagram are surjective.

\begin{example}
  The standard example is the absolute Galois group $G_{k}$ of a field
$k$. One takes all finite Galois extensions $L/k$ inside an
algebraically closed field $\Omega$ containing $k$ in the sense that
one has a fixed imbedding $i: k \to \Omega$, and the Galois extensions
are specific field extensions $L=k(\alpha)$ of $k$ inside $\Omega$.

These are the objects of a right filtered category, for which the
morphisms $L \to L'$ are extensions inside $\Omega$. The contravariant
functor $G_{k}$ that associates the Galois group $G_{k}(L) = G(L/k)$ to each of
these extensions is the absolute Galois group.

It is a basic assertion of field theory that if $L \subset L'$ inside $\Omega$ which are finite Galois extensions, then every field automorphism $\alpha: L' \to L'$ that fixes $k$ also permutes the roots which define $L$ over $k$, and hence restricts to an automorphism $\alpha\vert_{L}: L \to L$. The assignment $\alpha \mapsto \alpha\vert_{L}$ determines a surjective group homomorphism $G(L'/k) \to G(L/k)$.
\end{example}

Let $G-\mathbf{Set}_{df}$ be the category of finite discrete $G$-sets,
as in \cite{GECT}.
A discrete $G$-set is a set $F$ equipped
with an action 
\begin{equation*}
G \times F \to G_{i} \times F \to F,
\end{equation*}
where we write $G = \varprojlim_{i}\ G_{i}$, and a
morphism of discrete $G$-sets is a $G$-equivariant map.

Every finite discrete $G$-set $X$ has the form
\begin{equation*}
  X = G/H_{1} \sqcup \dots \sqcup G/H_{n},
\end{equation*}
where the groups $H_{i}$ are stabilizers of elements of $X$. In this way, the category $G-\mathbf{Set}_{df}$ is a thickening of the orbit category $\mathcal{O}_{G}$ for the profinite group $G$, whose objects are the finite quotients $G/H$ with $G$-equivariant maps between them. The subgroups $H_{i}$ are special: they are preimages of subgroups of the $G_{i}$ under the maps $G \to G_{i}$.

\begin{example}
  The finite \'etale site $fet\vert_{k}$ of $k$ is a category of schemes which has as objects all finite disjoint unions
  \begin{equation*}
    \Sp(L_{1}) \sqcup \dots \sqcup \Sp(L_{n})
  \end{equation*}
  of schemes defined by finite separable extensions $L_{i}/k$. The morphisms of
  $fet\vert_{k}$ are the scheme homomorphisms
  \begin{equation*}
    \Sp(L_{1}) \sqcup \dots \sqcup \Sp(L_{n}) \to \Sp(N_{1}) \sqcup \dots \sqcup \Sp(N_{m})
  \end{equation*}
over $k$,  or equivalently $k$-algebra homomorphisms
  \begin{equation*}
  \prod_{j}\ N_{j} \to \prod_{i}\ L_{i}.
\end{equation*}

  A finite separable extension $N=k(\alpha)$ of $k$ is specified by the root $\alpha$  in $\Omega$ of some separable polynomial $f(x)$.  A $k$-algebra map $N \to \Omega$ is specified by a root of $f(x)$ in $\Omega$, albeit not uniquely.

  One finds a finite Galois extension $L$ of $N$ by adjoining all roots of $f(x)$ to $N$. Then $L/N$ is Galois with Galois group $H=G(L/N)$, which is a subgroup of $G=G(L/k)$. The set of distinct maps $N \to \Omega$ can be identified with the set $G/H$, and $L$ is the fixed field of $H$.

It follows that there is a one-to-one correspondence
  \begin{equation*}
    \{ \text{finite separable $N/k$} \} \leftrightarrow \{ \text{$G_{k}$-sets $G/H$, $G=G(L/k)$ finite, $H \leq G$} \}
  \end{equation*}
  This correspondence determines an isomorphism of categories
  \begin{equation*}
    fet\vert_{k} \cong G_{k}-\mathbf{Set}_{df}.
  \end{equation*}

  If $k \subset L \subset N$ are finite separable extensions, then the function
  \begin{equation*}
    \hom_{k}(N,\Omega) \to \hom_{k}(L,\Omega)
  \end{equation*}
  is surjective, while the scheme homomorphism $\Sp(N) \to \Sp(L)$ is an \'etale cover. 
  \end{example}

For a general profinite group $G$, the category $G-\mathbf{Set}_{df}$ has a Grothendieck topology for which the
covering families are the $G$-equivariant surjections $U \to V$.

A presheaf $F$ is a sheaf for this topology if and only if $F(\emptyset)$ is a point, and every surjection $\phi = (\phi_{i}): \sqcup\ U_{i} \to V$ (covering family) induces an equalizer
\begin{equation}\label{eq 10}
  F(V) \to \prod_{i} \ F(U_{i}) \rightrightarrows \prod_{i \ne j}\ F(U_{i} \times_{V} U_{j}).
\end{equation}

The resulting sheaf category 
\begin{equation*}
\mathcal{B}G := \Shv(G-\mathbf{Set}_{df})
\end{equation*}
is often called the {\it classifying topos} for the
profinite group $G$.

\begin{lemma}\label{lem 5}
A
presheaf $F$ on $G-\mathbf{Set}_{df}$ is a sheaf if
and only if
\begin{itemize}
\item[1)]
$F$ takes disjoint unions to products, and
\item[2)]
each canonical map $G_{i} \to G_{i}/H$ induces a bijection
\begin{equation*}
F(G_{i}/H) \xrightarrow{\cong} F(G_{i})^{H}.
\end{equation*}
\end{itemize}
\end{lemma}

\noindent
The assertion that a presheaf $F$ takes disjoint unions to products is often called the {\it additivity} condition for $F$.

\begin{proof}
If $F$ is a sheaf, then the covering given by the inclusions $U \to U \sqcup V$ and $V \to U \sqcup V$ defines an isomorphism
\begin{equation*}
  F(U \sqcup V) \xrightarrow{\cong} F(U) \times F(V),
\end{equation*}
since $U \times_{U \sqcup V} V = \emptyset$. Also, if $H \subset G_{i}$ is a subgroup then $G_{i} \times_{G_{i}/H} G_{i} \cong \sqcup_{H}\ G_{i}$, and the coequalizer
\begin{equation*}
  F(G_{i}/H) \to F(G_{i}) \rightrightarrows \prod_{H}\ F(G_{i})
\end{equation*}
identifies $F(G_{i}/H)$ with the set of $H$-invariants $F(G_{i})^{H}$.

Conversely, if the presheaf $F$ satisfies conditions 1) and 2) and $G_{i}/K \to G_{i}/H$ is an equivariant map, then $K$ is conjugate to a subgroup of $H$, so we can assume that $K \subset H$ up to isomorphism. Then $F(G_{i})^{H}$ is isomorphic to the set of $H$-invariants of $F(G_{i})^{K}$, so that the equivariant covering $G_{i}/K \to G_{i}/H$ defines an equalizer of the form (\ref{eq 10}).
\end{proof}

It follows from Lemma \ref{lem 5} that every
discrete $G$-set $F$ represents a sheaf
\begin{equation*}
  F:= \hom(\ ,F)
\end{equation*}
  on $G-\mathbf{Set}_{df}$.
  \medskip
  
Let
\begin{equation*}
\pi: G-\mathbf{Set}_{df} \to \mathbf{Set}
\end{equation*}
be the functor which takes a finite discrete $G$-set to its underlying
set. Every set $X$ represents a sheaf $\pi_{\ast}X$ on
$G-\mathbf{Set}_{df}$ with
\begin{equation*}
\pi_{\ast}X(U) = \hom(\pi(U),X).
\end{equation*}
The left adjoint $\pi^{\ast}$ of the
corresponding functor $\pi_{\ast}$
has the form
\begin{equation*}
\pi^{\ast}F = \varinjlim_{i}\ F(G_{i}),
\end{equation*}
by a cofinality argument.
\medskip


A map $f: F \to G$ of presheaves is a {\it local epimorphism} if, given $y \in G(U)$ there is a covering $\phi: V \to U$ such that $\phi^{\ast}(y)$ is in the image of $f: F(V) \to G(V)$.

The presheaf map $f: F \to G$ is a {\it local monomorphism} if, given $x,y \in F(U)$ such that $f(x)=f(y)$ there is a covering $\phi: V \to U$ such that $\phi^{\ast}(x) = \phi^{\ast}(y)$ in $F(V)$.

It is a general fact that a morphism $f: F \to G$ of sheaves is an isomorphism if and only if it is both a local monomorphism and a local epimorphism.

Finally, one can show that a map $f: F \to G$ of sheaves on
$G-\mathbf{Set}_{df}$ is a local epimorphism (respectively local
monomorphism) if and only if the induced function
\begin{equation*}
  \pi^{\ast}(f): \pi^{\ast}(F) \to \pi^{\ast}(G)
\end{equation*}
is surjective (respectively injective). It follows that $f$ is an
isomorphism if and only if the function $\pi^{\ast}(f)$ is bijective.

We have a functor $\pi^{\ast}$ which is both exact (i.e. preserves
finite limits) and is faithful. This means that the corresponding
geometric morphism
\begin{equation*}
  \pi = (\pi^{\ast},\pi_{\ast}): \Shv(G-\mathbf{Set}_{df}) \to \mathbf{Set}
\end{equation*}
is a stalk (or Boolean localization) for the category of sheaves and presheaves on $G-\mathbf{Set}_{df}$, and gives a complete description of the local behaviour of sheaves and presheaves on this site.
\medskip

We use these observations to start up a homotopy theoretic machine \cite{LocHom}.
A map $f: X \to Y$ of simplicial presheaves (or simplicial sheaves) on
$G-\mathbf{Set}_{df}$ is a {\it local weak equivalence} if and only if the
induced map $\pi^{\ast}X \to \pi^{\ast}Y$ is a weak equivalence of
simplicial sets. 

The local weak equivalences are the weak equivalences of the 
{\it injective model structure} on the simplicial presheaf category for
the site $G-\mathbf{Set}_{df}$. The {\it cofibrations} are the
monomorphisms of simplicial presheaves (or simplicial sheaves). The fibrations for this
structure, also called the {\it injective fibrations}, are the maps
which have the right lifting property with respect to all cofibrations
which are local weak equivalences.

There are two model structures here, for the category $s\Pre(G-\mathbf{Set}_{df})$ of simplicial presheaves and for the category $s\Shv(G-\mathbf{Set}_{df})$ of simplicial sheaves, respectively. The forgetful and associated sheaf functors determine an adjoint pair of functors
\begin{equation*}
  L^{2}: s\Pre(G-\mathbf{Set}_{df}) \leftrightarrows s\Shv(G-\mathbf{Set}_{df}): u,
\end{equation*}
which is a Quillen equivalence, essentially since the canonical associated sheaf map $\eta: X \to L^{2}X = uL^{2}(X)$ is a local isomorphism, and hence a
local weak equivalence.

The associated sheaf functor $L^{2}$ usually has a rather formal construction, but in this case there is a nice description:
\begin{equation*}
  L^{2}F(G_{i}/H) = \varinjlim_{G_{j} \to G_{i}}\ F(G_{j})^{p^{-1}(H)},
\end{equation*}
where $p: G_{j} \to G_{i}$ varies over the transition maps of $G$ which take values in $G_{i}$.
\medskip

Here's a trick: suppose that $f: E \to F$ is a function, and form the groupoid
$E/f$ whose objects are the elements of $E$, and such that there is a unique morphism $x \to y$ if $f(x) = f(y)$. The corresponding nerve $B(E/f)$ has contractible path components, since each path component is the nerve of a trivial groupoid, and there is an isomorphism $\pi_{0}B(E/f) \cong f(E)$. It follows that there are simplicial set maps
\begin{equation*}
  B(E/f) \xrightarrow{\simeq} f(E) \subset F,
\end{equation*}
where the sets $f(E)$ and $F$ are identified with discrete simplicial
sets.  In particular, if $f$ is surjective then the map $B(E/f) \to F$
is a weak equivalence.

This construction is functorial, and hence applies to presheaves and sheaves. In particular, suppose that $\phi: V \to U$ is a local epimorphism of presheaves. Then the simplicial presheaf map
\begin{equation*}
  \check{C}(V):= B(E/\phi) \to U
\end{equation*}
is a local weak equivalence of simplicial presheaves, because $B(E/\phi) \to \phi(V)$ is a sectionwise hence local weak equivalence and $\phi(V) \to U$ induces an isomorphism of associated sheaves.

As the notation suggests, $\check{C}(V)$ is the \v{C}ech resolution for the covering $\phi$. All \v{C}ech resolutions arise from this construction.
\medskip

\noindent
    {\bf Examples}:\ 1)\ Suppose that $G= \{ G_{i} \}$ is a profinite group.
    The one-point set $\ast$ is a terminal object of the category
$G-\mathbf{Set}_{df}$. The group $G_{i}$ defines a covering $G_{i}=\hom(\ ,G_{i}) \to
    \ast$ of the terminal object, while the group $G_{i}$ acts on the sheaf $G_{i}=\hom(\ ,G_{i})$ by composition. There is a simplicial presheaf map
    \begin{equation*}
      \eta: EG_{i} \times_{G_{i}} G_{i} \to \check{C}(G_{i})\
    \end{equation*}
    which takes a morphism $\phi \to g\cdot\phi$ to the pair $(\phi,g\cdot\phi)$. The map $\eta$ induces an isomorphism in sections corresponding to quotients $G_{j}/H$ for $j \geq i$, hence in stalks, and is therefore the associated sheaf map and a local weak equivalence. 
    \medskip

    \noindent
    2)\ Suppose that $L/k$ is a finite Galois extension with Galois group $G$, and let $\Sp(L) \to \ast$ be the corresponding sheaf epimorphism on the finite \'etale site for $k$. The Galois group $G$ acts on $\Sp(L)$, and there is a canonical map
    \begin{equation*}
      \eta: EG \times_{G} \Sp(L) \to \check{C}(L).
      \end{equation*}

    For a finite separable extension $N/k$, the sections $\Sp(L)(N)$ are the $k$-algebra maps $L \to N$. Any two such maps determine a commutative diagram
    \begin{equation*}
      \xymatrix@R=8pt{
        L \ar[dr] \ar[dd]_{\sigma} \\
        & N \\
        L \ar[ur]
      }
    \end{equation*}
  where $\sigma$ is a uniquely determined element of the Galois group $G$. It follows that $\eta$ is an isomorphism in sections corresponding to all such extensions $N$, and so $\eta$ is the associated sheaf map for the simplicial presheaf $EG \times_{G} \Sp(L)$, and is a local weak equivalence of simplicial presheaves for the \'etale topology.
  \medskip

  \begin{remark}\label{rem 6}
  Every category of simplicial presheaves has an auxiliary model structure which is defined by cofibrations as above and sectionwise weak equivalences.  A map $X \to Y$ is a sectionwise weak equivalence if all induced maps $X(U) \to Y(U)$ in sections are weak equivalences of simplicial sets for all objects $U$ in the underlying site. This model structure is a special case of the injective model structure for simplicial presheaves on a site, for the so-called chaotic topology \cite[Ex. 5.10]{LocHom}.

  The fibrations for this model structure will be called {\it injective fibrations of diagrams} in what follows. These are the maps which have the right lifting property with respect to all cofibrations $A \to B$ which are sectionwise weak equivalences.

  Every injective fibration of simplicial presheaves is an injective fibration of diagrams, since every sectionwise weak equivalence is a local weak equivalence. The converse is not true.
  \end{remark}

  In all that follows, an {\it injective fibrant model} of a simplicial presheaf $X$ is a local weak equivalence $j: X \to Z$ such that $Z$ is injective fibrant.

  Every simplicial presheaf $X$ has an injective fibrant model: factorize the canonical map $X \to \ast$ to the terminal object as a trivial cofibration $j: X \to Z$, followed by an injective fibration $Z \to \ast$.

  Here is an example: if $F$ is a presheaf, identified with a simplicial presheaf which is discrete in the simplicial direction, then the associated sheaf map $\eta: F \to \tilde{F}$ is an injective fibrant model.

  Any two injective fibrant models of a fixed simplicial presheaf $X$ are equivalent in a very strong sense --- they are homotopy equivalent.

  In effect, every local weak equivalence $Z_{1} \to Z_{2}$ of injective fibrant objects is a homotopy equivalence, for the cylinder object that is defined by the standard $1$-simplex $\Delta^{1}$. It follows that all simplicial set maps $Z_{1}(U) \to Z_{2}(U)$ are homotopy equivalences.

  In particular, every local weak equivalence of injective fibrant objects is a sectionwise equivalence. 
  \medskip
  
  The injective model structure on the simplicial presheaf category $s\Pre(G-\mathbf{Set}_{df})$ is a simplicial model structure, where the function complex $\mathbf{hom}(X,Y)$ has $n$-simplices given by the maps $X \times \Delta^{n} \to Y$.

All simplicial presheaves are cofibrant. It follows that, if $Z$ is an injective fibrant simplicial presheaf and the map $\theta: A \to B$ is a local weak equivalence, then the induced map
\begin{equation*}
  \theta^{\ast}: \mathbf{hom}(B,Z) \to \mathbf{hom}(A,Z)
\end{equation*}
is a weak equivalence of simplicial sets.

\section{Cosimplicial spaces}

We shall use the Bousfield-Kan model structure for cosimplicial spaces
\cite{BK}, \cite{cosimp}.  The weak equivalences for this structure
are defined sectionwise: a map $f: X \to Y$ is a weak equivalence
of cosimplicial spaces if and only if all maps $X^{n} \to Y^{n}$ are
weak equivalences of simplicial sets.  The fibrations for the structure
are those maps $p: X \to Y$ for which all induced maps
\begin{equation*}
(p,s): X^{n+1} \to Y^{n+1} \times_{M^{n}Y} M^{n}X
\end{equation*}
are fibrations of simplicial sets.
Recall that $M^{n}X$ is the subcomplex of $\prod_{j=0}^{n} X^{n}$
which consists of those elements $(x_{0}, \dots ,x_{n})$ such that
$s^{j}x_{i} = s^{i}x_{j+1}$ for $i \leq j$,
and the canonical map $s: X^{n+1} \to M^{n}X$ is defined by
\begin{equation*}
s(x) = (s^{0}x,s^{1}x, \dots, s^{n}x).
\end{equation*}

The total complex $\Tot(X)$ for a fibrant cosimplicial space $X$ is
defined by
\begin{equation*}
\Tot(X) = \mathbf{hom}(\Delta,X),
\end{equation*}
where $\mathbf{hom}(\Delta,X)$ is the standard presheaf-theoretic
function complex, and $\Delta$ is the cosimplicial space of
standard simplices, given by the assignments $\mathbf{n} \mapsto \Delta^{n}$. The $p$-simplices of $\mathbf{hom}(\Delta,X)$ are the cosimplicial space maps $\Delta \times \Delta^{p} \to X$.
\medskip

If $X$ and $U$ are simplicial presheaves, write
$\hom(U_{\bullet},X)$ for the cosimplicial space $\mathbf{n}
\mapsto \hom(U_{n},X)$. If $U$ is representable by a
simplicial object $U$ in the underlying site, then
$\hom(U_{\bullet},X)$ can be identified up to isomorphism with
the cosimplicial space $\mathbf{n} \mapsto X(U_{n})$.

There are adjunction isomorphisms
\begin{equation*}
\begin{aligned}
\hom(\Delta \times \Delta^{p},\hom(U_{\bullet},X))
&\cong \hom(U,\mathbf{hom}(\Delta^{p},X))\\
&\cong \hom(U \times \Delta^{p},X),
\end{aligned}
\end{equation*}
which relate cosimplicial space maps to simplicial set maps.
Letting $p$ vary gives a natural isomorphism
\begin{equation*}
\Tot(\hom(U_{\bullet},X)) = \mathbf{hom}(\Delta,\hom(U_{\bullet},X)) 
\cong \mathbf{hom}(U,X)
\end{equation*}
of simplicial sets, for all simplicial presheaves $U$ and $X$.

\begin{lemma}\label{lem 7}
Suppose that $U$ is a simplicial presheaf. Then the functor $X \mapsto
\hom(U_{\bullet},X)$ takes injective fibrations of diagrams to
Bousfield-Kan fibrations of cosimplicial spaces.
\end{lemma}

\begin{proof}
There is an isomorphism
\begin{equation*}
M^{n}\hom(U_{\bullet},X) \cong \hom(DU_{n+1},X^{n+1}),
\end{equation*}
where $DU_{n+1} \subset U_{n+1}$ is the degenerate part of $U_{n+1}$
in the presheaf category. Suppose that $p: X \to Y$ is an injective
fibration. Then $p$ has the right lifting property with respect to the
trivial cofibrations
\begin{equation*}
(U_{n+1} \times \Lambda^{m}_{k}) \cup (DU_{n+1} \times \Delta^{m}) 
\subset U_{n+1} \times \Delta^{m},
\end{equation*}
so that the map
\begin{equation*}
\hom(U_{n+1},X) \to 
\hom(U_{n+1},Y) \times_{\hom(DU_{n+1},Y)} \hom(DU_{n+1},X)
\end{equation*}
is a fibration.
\end{proof}

\begin{remark}
If $Y$ is a Bousfield-Kan fibrant cosimplicial space then
there is a weak equivalence
\begin{equation*}
\Tot Y = \mathbf{hom}(\Delta,Y) \simeq \hoinvlim_{n}\ Y^{n},
\end{equation*}
which is natural in $Y$.

This is most easily seen
by using the injective model structure for cosimplicial spaces (i.e. for cosimplicial diagrams) of Remark \ref{rem 6} (see also
\cite{cosimp}).

In effect, if $j: Y \to Z$ is an injective fibrant model for $Y$ in
cosimplicial spaces, then $j$ is a weak equivalence of Bousfield-Kan
fibrant objects, so the map $j_{\ast}: \Tot(Y) \to \Tot(Z)$ is a weak
equivalence. It follows that there is a natural string of weak equivalences
\begin{equation*}
  \Tot(Y) \xrightarrow{\simeq} \Tot(Z) = \mathbf{hom}(\Delta,Z) \xleftarrow{\simeq} \mathbf{hom}(\ast,Z) = \varprojlim\ Z =: \hoinvlim\ Y,
  \end{equation*}
since the cosimplicial space $\Delta$ is cofibrant for the Bousfield-Kan structure.

It follows from Lemma \ref{lem 7} that if $U$ and $Z$ are simplicial
presheaves such that $Z$ is injective fibrant, then there is a natural
weak equivalence
\begin{equation*}
\mathbf{hom}(U,Z) \simeq \hoinvlim_{n}\ \hom(U_{n},Z).
\end{equation*}

\noindent
{\bf Examples}:\ 1)\
Suppose that $L/k$ is a finite Galois extension with Galois group $G$, and let $Y$ be a presheaf of Kan complexes for the finite \'etale site over $k$.
The function complex
\begin{equation*}
  \mathbf{hom}(EG \times_{G} \Sp(L),Y)
\end{equation*}
  can be rewritten as a homotopy inverse limit
\begin{equation*}
  \hoinvlim_{n}\ Y(\sqcup_{G^{\times n}}\ \Sp(L)) = \hoinvlim_{n} (\prod_{G^{\times n}}\ Y(L)) = \hoinvlim_{G}\ Y(L) = Y(L)^{hG},
\end{equation*}
which is the homotopy fixed points space for the action of $G$ on the space $Y(L)$ of $L$-sections of $Y$.
\medskip

\noindent
2)\ Similarly, if $G = \{ G_{i} \}$ is a profinite group and $X$ is a presheaf of Kan complexes on $G-\mathbf{Set}_{df}$, then
\begin{equation*}
\mathbf{hom}(EG_{i} \times_{G_{i}} G_{i},X) \simeq \hoinvlim_{G_{i}} 
\ X(G_{i}) = X(G_{i})^{hG_{i}}
\end{equation*}
is the homotopy fixed points space for the action of the group $G_{i}$ on the space $X(G_{i})$.
\end{remark}

We prove the following:

\begin{theorem}\label{th 9}
Suppose that $f: X \to Y$ is a local weak equivalence between
presheaves of Kan complexes on the site $G-\mathbf{Set}_{df}$ such that $X$
and $Y$ are sectionwise $n$-types. Then the induced map
\begin{equation*}
f_{\ast}: \varinjlim_{i}\ \mathbf{hom}(EG_{i} \times_{G_{i}} G_{i},X) \to 
\varinjlim_{i}\ \mathbf{hom}(EG_{i} \times_{G_{i}} G_{i},Y)
\end{equation*}
is a weak equivalence.
\end{theorem}

\begin{corollary}\label{cor 10}
Suppose that $X$ is a presheaf of Kan complexes on $G-\mathbf{Set}_{df}$
that is a sectionwise $n$-type, and let $j:
X \to Z$ be an injective fibrant model. Then the induced map of simplicial sets
\begin{equation*}
\varinjlim_{i}\ \mathbf{hom}(EG_{i} \times_{G_{i}} G_{i},X) \underset{j_{\ast}}{\xrightarrow{\simeq}}
\varinjlim_{i}\ \mathbf{hom}(EG_{i} \times_{G_{i}} G_{i},Z)
\end{equation*}
is a weak equivalence.
  \end{corollary}

\begin{proof}
If all presheaves of homotopy groups $\pi_{i}X$ are trivial for $i
\geq N$, then the homotopy groups $\pi_{i}Z$ are trivial for $i \geq
N$.

This is a special case of a very general fact
\cite[Prop 6.11]{GECT}. The proof uses a Postnikov tower argument, together with
the following statements:
\begin{itemize}
  \item[1)]
If $Y$ is a presheaf that is identified with a discrete simplicial sheaf, then the associated sheaf map $j: Y \to \tilde{Y}$ is an injective fibrant model.
\item[2)]
  If $Y = K(A,n)$ for some presheaf of groups $A$ and $j: K(A,n) \to Z$ is a fibrant model, then there are isomorphisms
  \begin{equation*}
    \pi_{j}(Z(U)) \cong
    \begin{cases}
      H^{n-j}(U,\tilde{A}\vert_{U}) & \text{if $0 \leq j \leq n$, and} \\
      0 & \text{if $j>n$.}
    \end{cases}
    \end{equation*}
  \end{itemize}
  
It follows from 
Theorem \ref{th 9} that the map $j_{\ast}$ is a weak equivalence.
\end{proof}

\begin{proof}[Proof of Theorem \ref{th 9}]
We can suppose that $X$ and $Y$ are injective fibrant {\it as diagrams} on
$G-\mathbf{Set}_{df}$ and that $f: X \to Y$ is an injective fibration
of diagrams. By Lemma \ref{lem 7}, all induced maps
\begin{equation*}
f: \hom((EG_{i} \times_{G_{i}} G_{i})_{\bullet},X) \to \hom((EG_{i} \times_{G_{i}} G_{i})_{\bullet},Y)
\end{equation*}
are Bousfield-Kan fibrations of Bousfield-Kan fibrant cosimplicial
spaces, and we want to show that the induced map
\begin{equation*}
\varinjlim_{i} \Tot \hom((EG_{i} \times_{G_{i}} G_{i})_{\bullet},X) 
\to \varinjlim_{i} \Tot \hom((EG_{i} \times_{G_{i}} G_{i})_{\bullet},Y)
\end{equation*}
is a trivial fibration of simplicial sets.

The idea is to show that all lifting problems
\begin{equation*}
\xymatrix{
\partial\Delta^{n} \ar[r] \ar[d] 
& \Tot \hom((EG_{i} \times_{G_{i}} G_{i})_{\bullet},X) \ar[d]^{f} \\
\Delta^{n} \ar[r] \ar@{.>}[ur] & \Tot \hom((EG_{i} \times_{G_{i}} G_{i})_{\bullet},Y)
}
\end{equation*}
can be solved in the filtered colimit. This is equivalent to the
solution of all cosimplicial space lifting problems
\begin{equation*}
\xymatrix{
& \hom((EG_{i} \times_{G_{i}} G_{i})_{\bullet},X^{\Delta^{n}}) \ar[d] \\
\Delta \ar[r] \ar@{.>}[ur] & 
\hom((EG_{i} \times_{G_{i}} G_{i})_{\bullet}, 
X^{\partial\Delta^{n}} \times_{Y^{\partial\Delta^{n}}} Y^{\Delta^{n}})
}
\end{equation*}
in the filtered colimit.

The induced map
\begin{equation*}
X^{\Delta^{n}} \to X^{\partial\Delta^{n}} \times_{Y^{\partial\Delta^{n}}} Y^{\Delta^{n}}
\end{equation*}
is an injective fibration of injective fibrant diagrams which is a
local weak equivalence, between objects which are sectionwise $n$-types.
It therefore suffices to show that all lifting problems
\begin{equation}\label{eq 11}
\xymatrix{
& \hom((EG_{i} \times_{G_{i}} G_{i})_{\bullet},X) \ar[d]^{f} \\
\Delta \ar[r]_-{\alpha} \ar@{.>}[ur] & \hom((EG_{i} \times_{G_{i}} G_{i})_{\bullet},Y)
}
\end{equation}
can be solved in the filtered colimit, for maps $f: X \to Y$ which are
locally trivial injective fibrations of diagrams between injective
fibrant objects, which objects have only finitely many non-trivial
presheaves of homotopy groups.

Suppose that $p: Z \to W$ is a locally trivial fibration of simplicial
presheaves on $G-\mathbf{Set}_{df}$, and suppose given a lifting problem
\begin{equation}\label{eq 12}
\xymatrix{
\partial\Delta^{n} \ar[r] \ar[d] & Z((EG_{i} \times_{G_{i}} G_{i})_{n}) \ar[d]^{p} \\
\Delta^{n} \ar[r] \ar@{.>}[ur] & W((EG_{i} \times_{G_{i}} G_{i})_{n})
}
\end{equation} 
There is a surjection
\begin{equation*}
  U \to (EG_{i} \times_{G_{i}} G_{i})_{n} = \sqcup_{G_{i}^{\times n}}\ G_{i}
\end{equation*}
  of finite discrete $G$-sets 
such that a lift exists in the diagram
\begin{equation*}
\xymatrix{
\partial\Delta^{n} \ar[r] \ar[d] & Z((EG_{i} \times_{G_{i}} G_{i})_{n}) \ar[r] & Z(U) \ar[d]^{p} \\
\Delta^{n} \ar[r] \ar[urr] & W((EG_{i} \times_{G_{i}} G_{i})_{n}) \ar[r] & W(U)
}
\end{equation*} 
There is a transition morphism $\gamma: G_{j} \to G_{i}$ in the pro-group $G$ and a discrete $G$-sets morphism $\sqcup_{G_{j}^{\times n}}\ G_{j} \to U$ such
that the composite
\begin{equation*}
\sqcup_{G_{j}^{\times n}}\ G_{j} \to U \to \sqcup_{G_{i}^{\times n}}\ G_{i} 
\end{equation*}
is the $G$-sets homomorphism which is induced by $\gamma$. It follows that
the lifting problem (\ref{eq 12}) has a solution
\begin{equation*}
\xymatrix{
\partial\Delta^{n} \ar[r] \ar[d] & Z((EG_{i} \times_{G_{i}} G_{i})_{n}) \ar[r]^{\gamma^{\ast}} 
& Z((EG_{j} \times_{G_{j}} G_{j})_{n}) \ar[d]^{p} \\
\Delta^{n} \ar[r] \ar[urr] & W((EG_{i} \times_{G_{i}} G_{i})_{n}) \ar[r]_{\gamma^{\ast}} & W(EG_{j} \times_{G_{j}} G_{j})_{n})
}
\end{equation*} 
after refinement along the induced map $\gamma: EG_{j} \times_{G_{j}} G_{j} \to EG_{i} \times_{G_{i}} G_{i}$.

All induced maps
\begin{equation*}
f: \hom((EG_{i} \times_{G_{i}} G_{i})_{\bullet},X) \to \hom((EG_{i} \times_{G_{i}} G_{i})_{\bullet},Y)
\end{equation*}
are Bousfield-Kan fibrations of cosimplicial spaces by Lemma \ref{lem 7}, as is their
filtered colimit
\begin{equation*}
f_{\ast}: \varinjlim_{i} \hom((EG_{i} \times_{G_{i}} G_{i})_{\bullet},X) 
\to \varinjlim_{i} \hom((EG_{i} \times_{G_{i}} G_{i})_{\bullet},Y).
\end{equation*}
The map $f_{\ast}$ is a weak equivalence of cosimplicial spaces by the
previous paragraph, and is therefore a trivial fibration.

In general, solving the lifting problem
\begin{equation*}
  \xymatrix{
    & Z \ar[d]^{p} \\
    \Delta \ar[r]_{\alpha} \ar@{.>}[ur]  & W
  }
  \end{equation*}
    for a map of cosimplicial spaces $p: Z \to W$ amounts to inductively solving a sequence of lifting problems
    \begin{equation}\label{eq 13}
      \xymatrix{
        \partial\Delta^{n+1} \ar[r] \ar[d] & Z_{n+1} \ar[d] \\
        \Delta^{n+1} \ar[r] \ar@{.>}[ur]
         & Y_{n+1} \times_{M^{n}Y} M^{n}Z
      }
    \end{equation}

It follows from the 
paragraphs above that, given a number $N \geq 0$, there is a structure map
$\gamma: G_{j} \to G_{i}$ for the pro-group $G$, such that the lifting
problems (\ref{eq 13})
associated to lifting a specific map
\begin{equation*}
  \alpha: \Delta \to \hom((EG_{i} \times_{G_{i}} G_{i})_{\bullet},Y)
\end{equation*}
to the total space of the map
\begin{equation*}
f_{\ast}: \hom((EG_{i} \times_{G_{i}} G_{i})_{\bullet},X) 
\to \hom((EG_{i} \times_{G_{i}} G_{i})_{\bullet},Y)
\end{equation*}
have a simultaneous solution in $\hom((EG_{j} \times_{G_{j}} G_{j})_{\bullet},X)$ for $n \leq N$.

If $X$ and $Y$ are sectionwise $N$-types, then the cosimplicial spaces
\begin{equation*}
  \hom((EG_{j}\times_{G_{j}} G_{j},X)_{\bullet},X)\ \text{and}\ \hom((EG_{j}\times_{G_{j}} G_{j},X)_{\bullet},Y)
\end{equation*}
  are sectionwise $N$-types.
The obstructions to the lifting problem
(\ref{eq 13}) for $f_{\ast}$ lie in $\pi_{k}X(\sqcup_{G_{j}^{\times (k+1)}}\ G_{j})$ and in
\begin{equation*}
\pi_{k+1}((\sqcup_{G_{j}^{\times (k+1)}}\ G_{j}) \times_{M^{k}\hom((EG_{j} \times_{G_{j}}G_{j})_{\bullet}, Y)} M^{k}\hom((EG_{j} \times_{G_{j}}G_{j})_{\bullet},X)),
\end{equation*}
which groups are $0$ since $k \geq N$.

It follows that, given a lifting problem (\ref{eq 11}), there is a
structure homomorphism $\gamma: G_{j} \to G_{i}$ of the pro-group $G$
such that the problem (\ref{eq 11}) is solved over $G_{j}$ in the sense
that there is a commutative diagram
\begin{equation*}
\xymatrix@R=14pt{ 
  && \hom((EG_{j}\times_{G_{j}}G_{j})_{\bullet},X) \ar[dd]^{f} \\
  \\
\Delta \ar[r]_-{\alpha} \ar[uurr] 
& \hom((EG_{i} \times_{G_{i}} G_{i})_{\bullet},Y) \ar[r]_-{\gamma^{\ast}} 
& \hom((EG_{j}\times_{G_{j}}G_{j})_{\bullet},Y) 
}
\end{equation*}
\end{proof}

\begin{lemma}\label{lem 11}
Suppose that $p: X \to Y$ is a Bousfield-Kan fibration between
Bousfield-Kan fibrant cosimplicial spaces. Suppose that the diagrams $X$ and $Y$ are sectionwise $N$-types.
Then the spaces $M^{n}X
\times_{M_{n}Y} Y^{n+1}$ are $N$-types.
\end{lemma}

\begin{proof}
Recall that $M^{n}X = M^{n}_{n}X$, where $M^{n}_{p}X$ is the iterated
pullback of the maps $s^{i}: X^{n+1} \to X^{n}$ for $i \leq p$. Let
$s: X^{n+1} \to M^{n}_{p}X$ be the map $(s^{0},\dots ,s^{p})$.
There are natural pullback diagrams
\begin{equation*}
\xymatrix{
M^{n}_{p}X \ar[r] \ar[d] & X^{n} \ar[d]^{s} \\
M^{n}_{p-1}X \ar[r] & M^{n-1}_{p-1} X
}
\end{equation*}
The map $s: X^{n+1} \to M^{n}X = M^{n}_{n}X$ is a fibration, so that
all of the maps $s: X^{n} \to M^{n}_{p}X$ are
fibrations by an inductive argument.

Inductively, if the spaces $M^{n-1}_{p-1}X$ and $M^{n}_{p-1}X$
are $N$-types, then $M^{n}_{p}X$ is an $N$-type. It follows that the spaces $M^{n}X$ are $N$-types.

In the pullback diagram
\begin{equation*}
\xymatrix{
M^{n}X \times_{M^{n}Y} Y^{n+1} \ar[r] \ar[d] & Y^{n+1} \ar[d]^{s}  \\
M^{n}X \ar[r] & M^{n}Y
}
\end{equation*}
the map $s$ is a fibration and the spaces $Y^{n+1}$, $M^{n}Y$ and $M^{n}X$ are $N$-types.
Then it follows that the space $M^{n}X \times_{M^{n}Y} Y^{n+1}$
is an $N$-type.
\end{proof}

\section{Pro-objects}

Suppose that $X$ is a simplicial presheaf of Kan complexes on the site $G-\mathbf{Set}_{df}$ of discrete finite $G$-sets, where $G$ is a profinite group. 
Corollary \ref{cor 10} implies that the space
\begin{equation*}
  \varinjlim_{i}\ \mathbf{hom}(EG_{i} \times_{G_{i}} G_{i},X)
\end{equation*}
is weakly equivalent to the space $Z(\ast)$ of global sections of an injective fibrant model $Z$ of $X$, provided that the simplicial presheaf $X$ is a sectionwise $n$-type for some $n$. 

The main (and only) examples of sectionwise $n$-types are the (derived) finite Postnikov sections
\begin{equation*}
  \mathbf{P}_{n}Y = P_{n}\Ex^{\infty}Y
  \end{equation*}
of simplicial presheaves $Y$.
\medskip

If $Y$ happens to be a presheaf of Kan complexes, we skip the derived step and write
\begin{equation*}
  \mathbf{P}_{n}Y = P_{n}Y,
\end{equation*}
where $P_{n}Y$ is the classical Moore-Postnikov section construction \cite[VI.3]{GJ}.

It is a basic property of the Postnikov section construction that the map
\begin{equation*}
  p: Y(U) \to \mathbf{P}_{n}Y(U)
\end{equation*}
is a Kan fibration in each section, which induces isomorphisms
\begin{equation*}
  \pi_{k}(Y(U),x) \xrightarrow{\cong} \pi_{k}(\mathbf{P}_{n}Y(U),x)
\end{equation*}
for all vertices $x \in Y(U)$ and for $0 \leq k \leq n$. Furthermore $\pi_{k}(\mathbf{P}_{n}Y(U),x) = 0$ for all vertices $x$ and $k > n$.

The maps $p$ are arranged into a comparison diagram
\begin{equation*}
  \xymatrix@R=8pt{
    & \mathbf{P}_{n+1}Y \ar[dd]^{\pi} \\
    Y \ar[ur]^{p} \ar[dr]_{p} \\
    & \mathbf{P}_{n}Y
  }
  \end{equation*}
in which all maps are sectionwise Kan fibrations. The tower $\mathbf{P}_{\ast}Y$ of sectionwise fibrations is the Postnikov tower of the presheaf of Kan complexes $Y$.
\medskip

Recall that a {\it pro-object} in a category $\mathcal{C}$ is a functor $I \to \mathcal{C}$, where $I$ is a small left filtered category.

If $Y$ is a presheaf of Kan complexes, then the associated Postnikov tower $\mathbf{P}_{\ast}Y$ is a pro-object in simplicial presheaves.
\medskip

Every pro-object $E: I \to \mathcal{C}$ in a category $\mathcal{C}$
represents a functor $h_{E}: \mathcal{C} \to \mathbf{Set}$, with
\begin{equation*}
  h_{E}(X) = \varinjlim_{i}\ \hom(E_{i},X).
\end{equation*}

A {\it pro-map} $E \to F$ is a natural transformation $h_{F} \to h_{E}$. The pro-objects and pro-maps are the objects and morphisms of the category $pro-\mathcal{C}$, commonly called the pro category in $\mathcal{C}$.

Every object $Z$ in the category $\mathcal{C}$ is a pro-object, defined on the one-point category. A Yoneda Lemma argument shows that a natural transformation $h_{Z} \to h_{E}$ can be identified with an element of the filtered colimit
\begin{equation*}
  \varinjlim_{i}\ \hom(E_{i},Z),
\end{equation*}
  and we usually think of pro-maps $E \to Z$ in this way.

If $F: J \to \mathcal{C}$ is a pro-object and $i \in J$, then there is a pro-map $F \to F_{i}$ which is defined by the image of the identity on $F_{i}$ in the filtered colimit $\varinjlim_{j}\ \hom(F_{j},F_{i})$.

Any pro-map $\phi: E \to F$ can be composed with the canonical maps $F \to F_{i}$, and the map $\phi$ can then be identified with an element of the set
\begin{equation*}
  \varprojlim_{j}\ \varinjlim_{i}\ \hom(E_{i},F_{j}).
\end{equation*}

Every simplicial presheaf is a pro-object in simplicial
presheaves, and the derived Postnikov tower construction
\begin{equation*}
  \xymatrix@R=8pt{
    & \mathbf{P}_{n+1}X \ar[dd] \\
    X \ar[ur] \ar[dr] \\
    & \mathbf{P}_{n}X
  }
\end{equation*}
defines a natural pro-map $X \to \mathbf{P}_{\ast}X$.
\medskip

There is a hierarchy of model
structures for the category of pro-simplicial presheaves, which is
developed in \cite{pro}.

The ``base'' model structure is the Edwards-Hastings model structure,
for which a cofibration $A \to B$ is map that is isomorphic in the
pro category to a monomorphism in a category of diagrams. A weak
equivalence for this structure, an {\it Edwards-Hastings weak
  equivalence}, is a map $f: X \to Y$ of pro-objects (that are defined
on filtered categories $I$ and $J$, respectively) such that the
induced map of filtered colimits
\begin{equation*}
\varinjlim_{j \in J}\ \mathbf{hom}(Y_{j},Z) \to 
\varinjlim_{i \in I}\ \mathbf{hom}(X_{i},Z)
\end{equation*}
is a weak equivalence of simplicial sets for all injective fibrant
simplicial presheaves $Z$. 

Every pro-simplicial presheaf has a functorially defined Postnikov
tower $\mathbf{P}_{\ast}X$, which is again a pro-object, albeit with a
larger indexing category.

It is shown in \cite{pro} that the functor
$X \mapsto \mathbf{P}_{\ast}X$ satisfies the criteria for
Bousfield-Friedlander localization within the Edwards-Hastings model
structure, and thus behaves like stabilization of spectra. In
particular, one has a model structure for which a weak equivalence
(a {\it pro-equivalence}) is a map $X \to Y$ which induces an
Edwards-Hastings equivalence $\mathbf{P}_{\ast}X \to
\mathbf{P}_{\ast}Y$.  This is the {\it pro-equivalence structure} for
pro-simplicial presheaves. It has the same cofibrations as the
Edwards-Hastings structure.

The Edwards-Hastings structure and the pro-equivalence structure both specialize to model structures for pro-objects in simplicial sets. The special case of the Edwards-Hastings structure for simplicial sets was first constructed by Isaksen in \cite{Isaksen3} --- he calls it the strict model structure.
\medskip

We shall need the following:

\begin{lemma}\label{lem 12}
Suppose that the map $f: Z \to W$ of simplicial presheaves is a pro-equivalence. Then it is a local weak equivalence. 
\end{lemma}

\begin{corollary}\label{cor 13}
Suppose that the map $f: Z \to W$ of simplicial sets is a pro-equivalence. Then $f$ is a weak equivalence.
\end{corollary}

\begin{proof}[Proof of Lemma \ref{lem 12}]
The natural map $Z
\to \mathbf{P}_{\ast}Z$ induces an Edwards-Hastings weak equivalence
\begin{equation*}
\mathbf{P}_{n}Z \to \mathbf{P}_{n}\mathbf{P}_{\ast}Z
\end{equation*}
for all $n \geq 0$. The induced map
\begin{equation*}
\mathbf{P}_{n}\mathbf{P}_{\ast}Z \to \mathbf{P}_{n}\mathbf{P}_{\ast}W
\end{equation*}
is an Edwards-Hastings weak equivalence, since the
Postnikov section functors preserve Edwards-Hastings equivalences
(Lemma 25 of \cite{pro}). It follows that all simplicial presheaf maps
\begin{equation*}
\mathbf{P}_{n}Z \to \mathbf{P}_{n}W
\end{equation*}
are Edwards-Hastings weak equivalences, and hence local weak
equivalences of simplicial presheaves. This is true for all $n$, so the map $f: Z
\to W$ is a local weak equivalence.
\end{proof}

\section{Galois descent}

Suppose again that $G = \{ G_{i} \}$ is a profinite group, and let one of the groups $G_{i}$ represent a sheaf on the category $G-\mathbf{Set}_{df}$ of discrete finite modules.

Recall that the group $G_{i}$ acts on the sheaf $G_{i}$ which is represented by the $G$-set $G_{i}$, and the canonical map of simplicial sheaves $EG_{i} \times_{G_{i}} G_{i} \to \ast$ is a local weak equivalence, where $\ast$ is the terminal simplicial sheaf.

It follows that, if $Z$ is injective fibrant, then the induced map
\begin{equation*}
  \mathbf{hom}(\ast,Z) \to \mathbf{hom}(EG_{i} \times_{G_{i}} G_{i},Z)
\end{equation*}
between function complexes is a weak equivalence of simplicial sets. There is an identification $\mathbf{hom}(\ast,Z) = Z(\ast)$, so we have a weak equivalence
\begin{equation*}
  Z(\ast) \xrightarrow{\simeq} \mathbf{hom}(EG_{i} \times_{G_{i}} G_{i},Z)
\end{equation*}
between global sections of $Z$ and the homotopy fixed points for the action of $G_{i}$ on the simplicial set $Z(G_{i})$. This is the {\it finite descent property} for injective fibrant simplicial presheaves $Z$.

More generally, if $X$ is a presheaf of Kan complexes on $G-\mathbf{Set}_{df}$, we say that $X$ {\it satisfies finite descent} if the induced map
\begin{equation*}
  X(\ast) \to \mathbf{hom}(EG_{i} \times_{G_{i}} G_{i},X)
\end{equation*}
is a weak equivalence for each of the groups $G_{i}$ making up the profinite group $G$. We have just seen that all injective fibrant simplicial presheaves satisfy finite descent.
\medskip

Recall (from Section 1) that, if $f: Z \to W$ is a local weak equivalence between injective fibrant objects, then $f$ is a sectionwise equivalence.
It follows that any two injective fibrant models $j: X \to Z$ and $j':X \to Z'$ of a fixed simplicial presheaf $X$ are sectionwise equivalent.

One says that a simplicial presheaf $X$ {\it satisfies descent} if some (hence any) injective fibrant model $j: X \to Z$ is a sectionwise equivalence.

The general relationship between descent and finite descent is the following:

\begin{lemma}\label{lem 14}
Suppose that the presheaf of Kan complexes $X$ on $G-\mathbf{Set}_{df}$ satisfies descent. Then it satisfies finite descent.
\end{lemma}

\begin{proof}
  Take an injective fibrant model $j: X \to Z$, and form the diagram
  \begin{equation*}
    \xymatrix{
      X(\ast) \ar[r]^{j}_{\simeq} \ar[d] & Z(\ast) \ar[d]^{\simeq} \\
      \mathbf{hom}(EG_{i} \times_{G_{i}} G_{i},X) \ar[r]_{j_{\ast}} & \mathbf{hom}(EG_{i} \times_{G_{i}} G_{i},Z)
    }
    \end{equation*}
  The map $j_{\ast}$ coincides with the map
\begin{equation*}
  \hoinvlim_{G_{i}}\ X(G_{i}) \to \hoinvlim_{G_{i}}\ Z(G_{i})
\end{equation*}
of homotopy fixed point spaces which is defined by the $G_{i}$-equivariant weak equivalence $X(G_{i}) \to Z(G_{i})$, and is therefore a weak equivalence. It follows that the map
\begin{equation*}
  X(\ast) \to \mathbf{hom}(EG_{i}\times_{G_{i}}G_{i},X)
\end{equation*}
is a weak equivalence.
\end{proof}

There is a converse for Lemma \ref{lem 14}, for a simplicial presheaf which has only finitely many non-trivial presheaves of homotopy groups.
The following statement is a 
consequence of Corollary \ref{cor 10}:

\begin{corollary}\label{cor 15}
Suppose that $X$ is a presheaf of Kan complexes on $G-\mathbf{Set}_{df}$, and that $X$ is a sectionwise $n$-type for some $n$. Suppose that $X$ satisfies finite descent, and suppose that $j: X \to Z$ is an injective fibrant model. Then the map $j: X(\ast) \to Z(\ast)$ in global sections is a weak equivalence.
\end{corollary}

\begin{proof}
  Form the diagram
  \begin{equation}\label{eq 14}
    \xymatrix{
      X(\ast) \ar[r]^{j} \ar[d] & Z(\ast) \ar[d] \\
      \varinjlim_{i}\ \mathbf{hom}(EG_{i} \times_{G_{i}} G_{i},X) \ar[r]_{j_{\ast}}^{\simeq} & \varinjlim_{i}\ \mathbf{hom}(EG_{i} \times_{G_{i}} G_{i},Z)
    }
  \end{equation}
  The map $j_{\ast}$ is a weak equivalence by Corollary \ref{cor 10}.
  The maps
\begin{equation*}
  X(\ast) \to \mathbf{hom}(EG_{i}\times_{G_{i}}G_{i},X)\enskip \text{and}\enskip
  Z(\ast) \to \mathbf{hom}(EG_{i}\times_{G_{i}}G_{i},Z)
\end{equation*}
are weak equivalences since $X$ and $Z$ satisfy finite descent, so the vertical maps in the diagram (\ref{eq 14}) are weak equivalences. It follows that $j: X(\ast) \to Z(\ast)$ is a weak equivalence.
\end{proof}

The proof of Corollary \ref{cor 15} also implies the following:

\begin{corollary}\label{cor 16}
  Suppose that $X$ is a presheaf of Kan complexes on $G-\mathbf{Set}_{df}$, and that $X$ is a sectionwise $n$-type for some $n$. Suppose that $j: X \to Z$ is an injective fibrant model. Then the map $j: X(\ast) \to Z(\ast)$ is weakly equivalent to the map
  \begin{equation*}
    X(\ast) \to \varinjlim_{i}\ \mathbf{hom}(EG_{i}\times_{G_{i}}G_{i},X).
    \end{equation*}
  \end{corollary}

We can translate the finite descent concept to \'etale sites for fields: a presheaf of Kan complexes $X$ on the finite \'etale site $fet\vert_{k}$  of a field $k$ {\it satisfies finite descent} if, for any finite Galois extension $L/k$ with Galois group $G$, the local weak equivalence $EG \times_{G} \Sp(L) \to \ast$ induces a weak equivalence
\begin{equation}\label{eq 15}
  X(k) \to \mathbf{hom}(EG \times_{G} \Sp(L),X) = \hoinvlim_{G}\ X(L).
\end{equation}

\begin{remark}\label{rem 17}
  We have already seen arguments for the following statements:
  \smallskip
  
  \noindent
1)\ Every injective fibrant simplicial presheaf $Z$ on $fet\vert_{k}$ satisfies descent and satisfies finite descent.
\smallskip

\noindent
2)\ If a presheaf of Kan complexes $X$ on $fet\vert_{k}$ satisfies descent, then it satisfies finite descent.
\end{remark}

Theorem \ref{th 9} and its corollaries also translate directly.

\begin{theorem}\label{th 18}
Suppose that $f: X \to Y$ is a local weak equivalence between
presheaves of Kan complexes on the site $fet\vert_{k}$, and that $X$
and $Y$ are sectionwise $n$-types. Then the induced map
\begin{equation*}
f_{\ast}: \varinjlim_{L/k}\ \mathbf{hom}(EG \times_{G} \Sp(L),X) \to 
\varinjlim_{L/k}\ \mathbf{hom}(EG \times_{G} \Sp(L),Y)
\end{equation*}
is a weak equivalence.
\end{theorem}

The colimits in the statement of Theorem \ref{th 18} are indexed over
finite Galois extensions $L/k$ in the algebraic closure $\Omega$, with Galois groups
$G=G(L/k)$. Similar indexing will be used for all statements that
follow.

\begin{corollary}\label{cor 19}
Suppose that $X$ is a presheaf of Kan complexes on $fet\vert_{k}$, and
that $X$ is a sectionwise $n$-type. Let $j:
X \to Z$ be an injective fibrant model. Then the map $j$ induces a
weak equivalence
\begin{equation*}
j_{\ast}: \varinjlim_{L/k}\ \mathbf{hom}(EG \times_{G} \Sp(L),X) \to 
\varinjlim_{L/k}\ \mathbf{hom}(EG \times_{G} \Sp(L),Z).
\end{equation*}
\end{corollary}

\begin{corollary}\label{cor 20}
Suppose that $X$ is a presheaf of Kan complexes on $fet\vert_{k}$, and that $X$ is a sectionwise $n$-type. Suppose that $X$ satisfies finite descent, and that $j: X \to Z$ is an injective fibrant model. Then the map $j: X(k) \to Z(k)$ in global sections is a weak equivalence.
\end{corollary}

\begin{corollary}\label{cor 21}
  Suppose that $X$ is a presheaf of Kan complexes on $fet\vert_{k}$, and that $X$ is a sectionwise $n$-type. Suppose that $j: X \to Z$ is an injective fibrant model. Then the map $j: X(k) \to Z(k)$ is weakly equivalent to the map
  \begin{equation*}
    X(k) \to \varinjlim_{L/k}\ \mathbf{hom}(EG\times_{G} \Sp(L),X).
    \end{equation*}
  \end{corollary}
\bigskip

Now suppose that $X$ is a presheaf of Kan complexes on the finite \'etale site
$fet\vert_{k}$ of a field $k$. Let $\mathbf{P}_{n}X$ be the $n^{th}$ Postnikov section of $X$, with canonical map $p: X \to \mathbf{P}_{n}X$.

We construct a natural fibrant replacement $\eta: \mathbf{P}_{\ast}X \to L\mathbf{P}_{\ast}X$ for $\mathbf{P}_{\ast}X$ for the Postnikov tower in the category of towers of simplicial presheaves. This is done by inductively finding local weak
equivalences $\eta: \mathbf{P}_{n}X \to L\mathbf{P}_{n}X$ and injective
fibrations $q: L\mathbf{P}_{n}X \to L\mathbf{P}_{n-1}X$ such that the
diagrams
\begin{equation*}
\xymatrix{
\mathbf{P}_{n}X \ar[r]^{\eta} \ar[d]_{\pi} & L\mathbf{P}_{n}X \ar[d]^{q} \\
\mathbf{P}_{n-1}X \ar[r]_{\eta} & L\mathbf{P}_{n-1}X
}
\end{equation*}
commute.

Take an injective fibrant model $j: X \to Z$ for $X$, and form the diagram of simplicial set maps
\begin{equation}\label{eq 16}
  \xymatrix{
    \varinjlim_{L/k} \mathbf{hom}(EG \times_{G} \Sp(L),L\mathbf{P}_{n}X)
    \ar[r]^{j_{\ast}}
    & \varinjlim_{L/k} \mathbf{hom}(EG \times_{G} \Sp(L),L\mathbf{P}_{n}Z) \\
    \varinjlim_{L/k} \mathbf{hom}(EG \times_{G} \Sp(L),\mathbf{P}_{n}X) \ar[u]_{\simeq}^{\eta_{\ast}} \ar[r]^{j_{\ast}}_{\simeq}
& \varinjlim_{L/k} \mathbf{hom}(EG \times_{G} \Sp(L),\mathbf{P}_{n}Z) \ar[u]_{\eta_{\ast}}^{\simeq} \\
\mathbf{P}_{n}X(k) \ar[u]^{\alpha} \ar[r]^{j} & \mathbf{P}_{n}Z(k) \ar[u]_{\alpha} \\
X(k) \ar[u]^{p} \ar[r]_{j} & Z(k) \ar[u]_{p}
}
\end{equation}
The indicated maps are weak equivalences by Theorem \ref{th 18}. 

The diagram (\ref{eq 16}) can be interpreted as a commutative diagram
of pro-objects in simplicial sets, in which the maps $p$ are pro-equivalences. 
The vertical composites $\alpha \cdot p$ are the maps $\theta$ of the Introduction.

The weak equivalences
\begin{equation*}
  L\mathbf{P}_{n}Z(k) \xrightarrow{\simeq} \varinjlim_{L/k}\ \mathbf{hom}(EG \times_{G} \Sp(L),L\mathbf{P}_{n}Z)
\end{equation*}
give an equivalence of the vertical composite
\begin{equation*}
  \eta_{\ast}\cdot \alpha \cdot p: Z(k) \to \varinjlim_{L/k}\ \mathbf{hom}(EG \times_{G} \Sp(L),L\mathbf{P}_{n}Z)
\end{equation*}
with the pro-map $\eta \cdot p: Z(k) \to L\mathbf{P}_{\ast}Z(k)$.

\begin{lemma}\label{lem 22}
  Suppose that $\ell$ is a prime with $\ell \ne char(k)$.  Suppose
  that there is a uniform bound $N$ on the Galois cohomological
  dimension of $k$ with respect to $\ell$-torsion sheaves. Suppose that $Z$ is
  an injective fibrant object such 
  that each of the sheaves $\tilde{\pi}_{k}Z$ is $\ell^{m}$-torsion for
  some $m$. Then the map
  \begin{equation*}
    \eta \cdot p: Z(k) \to L\mathbf{P}_{\ast}Z(k)
  \end{equation*}
  is a pro-equivalence.
\end{lemma}

\begin{remark}
  The uniform bound assumption implies that if $L/k$ is any finite separable extension and $x \in Z(L)$ is a vertex, then $H^{p}_{et}(L,\tilde{\pi}_{k}(Z\vert_{L},x)) = 0$ for $p >N$. 
  Here, $Z\vert_{L}$ is the restriction of the simplicial presheaf $Z$ to the finite \'etale site of $L$.

  In effect, $\tilde{\pi}_{k}(Z\vert_{L},x)$ is an $\ell^{m}$-torsion sheaf, and the cohomological dimension of $L$ with respect to $\ell^{m}$-torsion sheaves is bounded above by that of $k$, by a Shapiro's Lemma argument \cite[Sec 3.3]{Serre-CG2}.

The existence of a global bound in Galois cohomological dimension of Lemma \ref{lem 22} is commonly met in practice, such as for the mod $\ell$ $K$-theory presheaves $(\mathbf{K}/\ell)^{n}$, when defined over fields $k$ that arise from finite dimensional objects and $\ell \ne 2$. See Thomason's paper \cite{AKTEC}.
\end{remark}

\begin{proof}[Proof of Lemma \ref{lem 22}] 
  All presheaves $\mathbf{P}_{n}Z$ have the same presheaf of vertices, namely $Z_{0}$, and there is a pullback diagram of simplicial presheaves
  \begin{equation*}
    \xymatrix{
      K(\pi_{n}Z,n) \ar[r] \ar[d] & \mathbf{P}_{n}Z \ar[d] \\
      Z_{0} \ar[r] & \mathbf{P}_{n-1}Z
    }
    \end{equation*}
  which defines the object $K(\pi_{n}Z,n)$. In sections, the fibre of the map
\begin{equation*}
  K(\pi_{n}Z,n)(U) \to Z_{0}(U)
\end{equation*}
  over the vertex $x \in Z_{0}(U)$ is the space $K(\pi_{n}(Z(U),x),n)$.

  Form the diagram
  \begin{equation*}
    \xymatrix{
    K(\pi_{n}Z,n) \ar[r]^{j}_{\simeq} \ar[d] & LK(\pi_{n}Z,n) \ar[d]^{q} \\
    Z_{0} \ar[r]_{j}^{\simeq} & \tilde{Z}_{0}
    }
    \end{equation*}
  where the maps labelled by $j$ are injective fibrant models and $q$ is an injective fibration. The fibrant model $j: Z_{0} \to \tilde{Z}_{0}$ can be identified with the associated sheaf map.

Suppose that $y \in LK(\pi_{n}Z,n)(k)_{0}$. There is a finite separable extension $L/k$ such that $q(y) \in \tilde{Z}_{0}(L)$ is in the image of the map $j: Z_{0}(L) \to \tilde{Z}_{0}(L)$, meaning that $q(y\vert_{L}) = j(z)$ for some $z \in Z_{0}(L)$.
  
  Form the pullback diagram
    \begin{equation*}
    \xymatrix{
      q^{-1}(q(y)) \ar[r] \ar[d] & LK(\pi_{n}Z,n) \ar[d]^{q} \\
      \ast \ar[r]_{q(y)} & \tilde{Z}_{0}
    }
    \end{equation*}
    Then
    \begin{equation*}
      \pi_{k}(q^{-1}(q(y))(k),y) = \pi_{k}(LK(\pi_{n}Z,n)(k),y).
    \end{equation*}
    The simplicial presheaf $q^{-1}(q(y))$ is injective fibrant, and has one non-trivial sheaf of homotopy groups, say $A$, in degree $n$. The sheaf $A$ is $\ell^{m}$-torsion, since its restriction to $fet\vert_{L}$ is the sheaf associated to the presheaf $\pi_{n}(Z\vert_{L},z)$, which is $\ell^{m}$-torsion.

    We therefore have isomorphisms
    \begin{equation*}
      \pi_{s}(LK(\pi_{n}Z,n)(k),y) = \pi_{s}(q^{-1}(q(y))(k),y) \cong
      \begin{cases}
        H^{n-s}_{et}(k,A) & \text{if $s \leq n$, and}\\
        0 & \text{otherwise}
      \end{cases}
    \end{equation*}
    (see \cite[Prop 8.32]{LocHom}, and the proof of Corollary \ref{cor 10}).
In particular, the homotopy groups $\pi_{s}(LK(\pi_{n}Z,n)(k),y)$ vanish for $s < n-N$.

It follows that the map
\begin{equation*}
    \varprojlim_{m}\ L\mathbf{P}_{m}Z(k) \to L\mathbf{P}_{n}Z(k)
  \end{equation*}
  induces a weak equivalence
\begin{equation*}
  \mathbf{P}_{r}(\varprojlim_{m}\ L\mathbf{P}_{m}Z(k)) \to \mathbf{P}_{r}(L\mathbf{P}_{n}Z(k))
\end{equation*}
for $n$ sufficiently large, and this is true for each $r$.

It also follows that the simplicial set map
\begin{equation*}
  Z(k) \to \varprojlim_{n}\ L\mathbf{P}_{n}Z(k)
\end{equation*}
is a weak equivalence, and that the map
\begin{equation*}
  \varprojlim_{n}\ L\mathbf{P}_{n}Z(k) \to L\mathbf{P}_{\ast}Z(k)
\end{equation*}
is a pro-equivalence.

The composite
\begin{equation*}
  Z(k) \to \varprojlim_{n}\ L\mathbf{P}_{n}Z(k) \to L\mathbf{P}_{\ast}Z(k)
\end{equation*}
is the map $\eta\cdot p$, and is a pro-equivalence.
\end{proof}

Thus, in the presence of a global bound on cohomological dimension as in Lemma \ref{lem 22}, we see that,
{\it with the
exception of} the maps $j: X(k) \to Z(k)$ and
\begin{equation*}
\alpha: \mathbf{P}_{n}X(k) \to 
\varinjlim_{L/k}\mathbf{hom}(EG \times_{G} \Sp(L),\mathbf{P}_{n}X),
\end{equation*}
the maps in
the diagram (\ref{eq 16}) are pro-equivalences.

The simplicial set
map $X(k) \to Z(k)$ is a weak equivalence if and only if it is a
pro-equivalence, by Lemma \ref{lem 12}. We have the
following consequence:

\begin{theorem}\label{th 24}
Suppose that $X$ is a presheaf of Kan complexes on the finite \'etale site
$fet\vert_{k}$ of a field $k$, such that the presheaves $\pi_{s}X$ are $\ell^{n}$-torsion for some $n$ and some prime $\ell \ne 2$, which is also distinct from the characteristic of $k$. Let $j: X \to Z$ be an injective
fibrant model of $X$. Suppose that there is a uniform bound on the
Galois cohomological dimension of $k$ for $\ell$-torsion sheaves.

Then the map
$j: X(k) \to Z(k)$ in global sections is a weak equivalence if and
only if the map of towers
\begin{equation*}
\alpha: \mathbf{P}_{n}X(k) \to 
\varinjlim_{L/k}\mathbf{hom}(EG \times_{G} \Sp(L),\mathbf{P}_{n}X)
\end{equation*}
is a pro-equivalence in simplicial sets.
\end{theorem}

\begin{remark}
The statement of Theorem \ref{th 24} is only an illustration. In geometric cases, one can refine the inclusion $k \subset k_{sep}$ of the field $k$ in its separable closure into a sequence of Galois subextensions
\begin{equation*}
  k=L_{0} \subset L_{1} \subset \dots \subset L_{N} = k_{sep}
\end{equation*}
such that each of the Galois extensions $L_{i+1}/L_{i}$ has Galois cohomological dimension $1$ with respect to $\ell$-torsion sheaves --- see Section 7.7 of \cite{GECT}. Then there is a statement analogous to Theorem \ref{th 24} for the finite Galois subextensions $L/L_{i}$ of $L_{i+1}/L_{i}$.

Historically, the use of this decomposition was meant to break up the problem of proving the Lichtenbaum-Quillen conjecture into proving descent statements in relative Galois cohomological dimension $1$. This attack on the conjecture was never successfully realized.
\end{remark}

\nocite{BrGe}

\vfill\eject

\bibliographystyle{plain} 
\bibliography{spt}

\end{document}